\documentclass[11pt,reqno]{amsart}
\usepackage{amssymb}
\usepackage{upgreek}
\usepackage{dsfont }
\usepackage{mathrsfs}
\usepackage{mathtools}
\usepackage[all]{xy}
\usepackage{color}
\usepackage{verbatim}
\usepackage[utf8]{inputenc}
\usepackage[T1]{fontenc}
\usepackage{amsmath}
\usepackage{amsfonts}
\usepackage[version=4]{mhchem}
\usepackage{stmaryrd}

\allowdisplaybreaks

\newtheorem{thm}{Theorem}[section]
\newtheorem{lem}[thm]{Lemma}

\theoremstyle{definition}

\newtheorem{ex}[thm]{Example}

\theoremstyle{remark}
\newtheorem{remark}[thm]{Remark}

\newtheorem{assumption}[thm]{Assumption}

\newcommand{\af}{\alpha}
\newcommand{\bt}{\beta}

\newcommand{\ep}{\varepsilon}

\newcommand{\om}{\omega}

\newcommand{\R}{{\mathbb{R}}}

\newcommand{\Hess}{\mathrm{Hess}}
\newcommand{\nab}{\nabla}
\newcommand{\pr}{\partial}

\begin{document}

\title[Monotonicity formulas and Hessian of the Green function]
{Monotonicity formulas and Hessian of the Green function}

\author[Jiewon Park]{Jiewon Park}
\address{
Department of Mathematical Sciences\\
Korea Advanced Institute of Science and Technology (KAIST)\\
291 Daehak-ro, Yuseong-gu, Daejeon 34141\\
South Korea} 
\email{jiewonpark@kaist.ac.kr}

\begin{abstract}   
Based on an assumption on the Hessian of the Green function, we derive some monotonicity formulas on nonparabolic manifolds. This assumption is satisfied on manifolds that meet certain conditions including bounds on the sectional curvature and covariant derivative of the Ricci curvature, as shown in the author's previous work \cite{P}. We also give explicit examples of warped product manifolds on which this assumption holds.
\end{abstract} 

\maketitle

\section{Introduction}

Numerous monotonicity formulas involving the Green function have been discovered in recent years, leading to new breakthroughs and novel proofs of significant theorems. An important example is the monotonicity of various functionals introduced by Colding in \cite{C} for nonnegative Ricci curvature, modeled on the area and volume functionals on the Euclidean space. These formulas were then extended by Colding-Minicozzi in \cite{CM2} to a 1-parameter family of monotonicity formulas. Such monotonicity formulas are extremely useful in that they provide tremendous geometric insights and implications. The aim of this paper is to provide more examples of them for a certain class of manifolds.

We recall that a Riemannian manifold is said to be {\it nonparabolic} if it admits a positive Green function for the Laplacian. For manifolds with Ricci curvature bounded below, faster-than-quadratic volume growth is a strong enough condition to yield nonparabolicity \cite{V}. 

Let $(M,g)$ be a complete nonparabolic Riemannian manifold of dimension $n \geq 3$ with nonnegative Ricci curvature. Fixing $p\in M$, we normalize the minimal positive Green function $G$ so that $\Delta G(p,\cdot) =-n(n-2) \om_n \delta_p$  where $\om_n$ is the volume of the unit $n$-ball. This normalization is such that $G(x,y) = |x-y|^{2-n}$ on the Euclidean space $\R^n$. Also, as in \cite{C} define the function $b:M\to [0, \infty)$ by $$b(x) = G(p,x)^{\frac{1}{2-n}}.$$ 

\begin{thm} \label{CM_monotonicity}
\cite[Corollary 2.21]{C} Define $A(r)$ by $$A(r) = r^{1-n} \int_{b=r} |\nab b|^3.$$ 

Then the following monotonicity formula holds.
$$A'(r)= -\frac{r^{n-3}}2\int_{b \geq r} \left(\left|\Hess_{b^2}-\frac{\Delta b^2}{n}g\right|^2 + \text{Ric}(\nabla b^2, \nabla b^2)\right) b^{2-2n}\, dV \leq 0.$$
\end{thm}

In addition to \cite{C, CM2} mentioned above, more monotonicity formulas were discovered by Agostiniani-Fogagnolo-Mazzieri \cite{AFM} also for nonnegative Ricci curvature. Many examples exist in other contexts as well. For bounded scalar curvature in dimension three, monotone quantities regarding the Green function were found by Agostiniani-Mazzieri-Oronzio \cite{AMO} and Munteanu-Wang \cite{MW, MW2}. The latter was extended to the $p$-Laplacian by Chan-Chu-Lee-Tsang \cite{CCLT}. Agostiniani-Mazzieri \cite{AM} found monotone quantities for asymptotically flat static metrics arising in general relativity.

These formulas have seen a number of important applications, including a new proof of the positive mass theorem \cite{AMO}, geometric inequalities such as Willmore-type inequalities on manifolds with nonnegative Ricci curvature \cite{AFM}, and extended Minkowski inequalities \cite{AFM2} by the means of effective monotonicity formulas for the $p$-capacitary potential on $\R^n$ minus a compact set. The monotonicity formula in \cite{CM2} implies that the level sets of the Green function are asymptotically umbilic. The monotone quantity $A(r)$ in Theorem \ref{CM_monotonicity} was used in \cite{CM3} to establish the uniqueness of tangent cone in the presence of a smooth cross section. The monotonicity formulas of \cite{MW, MW2} were applied by Chodosh-Li \cite{CL} to address Bernstein-type problems for minimal surfaces. 

While Theorem \ref{CM_monotonicity} is proved using integration by parts techniques and Bochner formula, another way to obtain monotonicity is via a pointwise Hessian estimate, often referred to as a matrix Li-Yau-Hamilton inequality. Monotonicity formulas obtained this way include, for example, the parabolic frequency  of Poon \cite{Po}. For more examples and a treatise of the general principle, we refer the reader to the comprehensive survey by Ni \cite{N2}. 

Therefore we introduce the following assumption on the Hessian of $G$, which we call Assumption ($\ast$).

\begin{assumption} \label{ass:ast}
A complete, noncompact, nonparabolic Riemannian manifold $(M,g)$ is said to satisfy Assumption ($\ast$) if there exists some $C \geq 0$ such that $\Hess_{b^2} \leq Cg$.
\end{assumption}     

Some remarks about Assumption $(\ast)$ are in order. Here and henceforth we denote the distance from $p$ by $r$. We first note that the Euclidean space $\R^n$ with $n\geq 3$ satisfies Assumption ($\ast$) with $C=2$. 

 In general, $\Hess_{b^2} \leq Cg$ always holds true for some sufficiently large $C$ on a small neighborhood of $p$. For instance \cite[Theorem 2.4]{MRS} implies that $\displaystyle \left|\Hess_G + n(2-n)r^{-n} dr^2 -(2-n)r^{-n}g\right| = o(1)$, which implies that for any $C$ sufficiently large, $\Hess_{b^2} \leq Cg$ on a small neighborhood of $p$. On the other extreme as $r \to \infty$, in general on noncompact manifolds little is known about the second-order asymptotics of $G$ although there are related works on Euclidean spaces \cite[Lemma 2.3]{Pog}. In \cite{P}, the author has shown that under curvature bounds and an assumption at infinity, Assumption ($\ast$) is satisfied on all of the manifold with $C$ determined to be the same as that of a small neighborhood of $p$.

\begin{thm} \cite[Theorem 1.1]{P} \label{thm:Hessest}
Suppose that $M$ has parallel Ricci curvature and $Rm(\nab b,V,\nab b, V)\geq 0$ for any vector field $V$ on $M$, where $Rm$ is the Riemannian curvature tensor. Suppose that $\Hess_{b^2} \leq C_1 g$ on a neighborhood of $p$ and $\Hess_{b^2} \leq C_2 g$ outside of a compact set for constants $C_1, C_2$, where $C_1\geq 10$ and $C_2$ can be arbitrarily larger than $C_1$. Then $\Hess_{b^2} \leq C_1 g$ on $M\backslash\{p\}$. \end{thm}

The proof of Theorem \ref{thm:Hessest} readily generalizes to the situation where $|Rm|\leq Kr^{-2}$ and $|\nab \mathrm{Ric}|\leq Lr^{-3}$ for some constants $K, L>0$, in which case the resulting bound on $\Hess_{b^2}$ also depends on not only $C_1$ but also $K, L$. We also note that Assumption ($\ast$) can be shown to hold on Einstein manifolds, as discussed in \cite[Theorem 4.1]{CM3}.

Based on Assumption ($\ast$), in this paper we prove two types of monotonicity formulas. The first family of monotonicity formulas regards the gradient flow of $b^2$. It has the advantage of being very flexible in the choice of the domain of integration. 

\begin{thm} \label{thm1}
Suppose that Assumption ($\ast$) holds true. Let $\Phi_t$ be the gradient flow of $b^2$. Given a compact smooth domain $B$ and constants $\alpha, \beta$ such that $\bt \geq 0$ and $\af + \bt \geq n$, 
$$\frac{d}{dt} \left(e^{-C\bt t} \int_{\Phi_{t}(B)} b^{-\alpha}|\nabla b|^{\beta} \right)\leq 0.$$ 

If equality holds at some $t$ and $p \in \Phi_t(B)$, then $M$ is isometric to the Euclidean space $\R^n$.
\end{thm}

Theorem \ref{thm1} seems to be the optimal form of the following somewhat more general theorem.

\begin{thm} \label{thm2}
Suppose that Assumption ($\ast$) holds true. For any compact smooth domain $B$ and constants $\af, \bt$ such that $\bt \geq 0$ and $\af + \bt \geq n$, and any differentiable function $f:\R \to \R$ satisfying $\displaystyle \frac{f'(t)}{f(t)} \leq -C\beta$, the following monotonicity formula holds true.
$$\frac{d}{dt} \left(f(t) \int_{\Phi_{t}(B)} b^{-\alpha}|\nabla b|^{\beta} \right)\leq 0.$$ If equality holds at some $t$ and $p \in \Phi_t(B)$, then $M$ is isometric to the Euclidean space $\R^n$.
\end{thm}

We additionally derive another monotonicity formula, which seems closer in spirit with \cite{C, CM2, AFM, MW} where the monotonicity formulas hold along level sets or sublevel sets.

\begin{thm} \label{thm3}
Under Assumption ($\ast$), we have
$$\frac{d}{dr}\left(r^{2-n}\int_{b\leq r}|\nab b|^4 - \frac{C|\mathbb{S}^{n-1}|}{2n} r^2 \right) \leq 0.$$ If equality holds at some $r>0$, then $M$ is isometric to the Euclidean space $\R^n$.
\end{thm}

This paper is organized as follows. In Section 2 we gather some preliminaries and prove Theorems \ref{thm1} and \ref{thm2}. The proof of Theorem \ref{thm3} is presented in Section 3. Some demonstrative examples are given in Section 4, where we explicitly verify Assumption $(\ast)$ on some warped products.

\section{Proof of the first monotonicity formula}

Let us first record some estimates and identities needed in the proof. Firstly, $\Delta G = 0$ is equivalent to $\Delta b^2 = 2n|\nab b|^2$. 

The following gradient estimate for $b$ was proved in \cite{C}.

\begin{thm}{\cite[Theorem 3.1]{C}} \label{thm:gradest}
$|\nab b| \leq 1$ on $M$. If equality holds at any point on $M$, then $M$ is isometric to the Euclidean space $\R^n$.
\end{thm}

We will also need the identity from \cite{C} that
 $$\int_{b \leq r} |\nab b|^2 = \frac{|\mathbb{S}^{n-1}|}{n} r^n$$ for any $r>0$, which can be derived using the coarea formula together with the asymptotics that $|\nab b| \to 1$ as one approaches $p$. We will also use the following lemma repeatedly, the proof of which is a straightforward calculation.

\begin{lem} \label{lem_nab_grad}
$\langle \nab |\nab b|^2, \nab b^2 \rangle = 2(\nab^2 b^2)(\nab b, \nab b) - 4|\nab b|^4$.
\end{lem}

\begin{proof}
In normal coordinates we get $(b^2)_{ij} = 2bb_{ij} + 2b_ib_j$. Equivalently, $$b_{ij} = \frac{(b^2)_{ij} - 2b_i b_j}{2b}.$$
Therefore,
$$
\begin{aligned}
\langle \nab |\nab b|^2, \nab b^2 \rangle &= ((b_i)^2)_j (b^2)_j = 4b \cdot b_{ij}b_ib_j\\
&=2 \left((b^2)_{ij} - 2b_i b_j\right) b_i b_j  \\
&= 2(\nab^2 b^2)(\nab b, \nab b) - 4|\nab b|^4.
\end{aligned}
$$

\end{proof}

Let us first explain the motivation behind monotonicity formulas like Theorem \ref{thm1} or Theorem \ref{thm2}. For this we observe the model case of $\R^n$ where $b$ is the distance from the origin and $B$ is the unit ball. Then the integral $ \int_{\Phi_{t}(B)} b^{-\alpha}|\nab b|^{\bt}$ is easily tractable since $\Phi_t(x) = e^{2t}x$ and $|\nab b| = 1$.

\begin{equation} \label{int_model}
	\int_{\Phi_{t}(B)} b^{-\alpha}|\nab b|^{\bt}=\int_{\left\{b \leq e^{2 t}\right\}} b^{-\alpha} =\left|\mathbb{S}^{n-1}\right|\int_{0}^{e^{2 t}} s^{n-1-\alpha} \, ds.
\end{equation}

The monotonicity of this quantity itself is immediate, since it is the integral of a positive function over a domain that becomes larger with $t$. However this monotonicity is not very interesting or sharp even in this model case, so we consider $f(t) \int_{\Phi_{t}(B)} b^{-\alpha}|\nabla b|^{\beta}$ for an appropriate decreasing positive function $f: (0,\infty) \to (0,\infty)$. Completing the calculation in (\ref{int_model}), in this model case we have for $\af \neq n$,

$$f(t) \int_{\Phi_{t}(B)} b^{-\alpha}|\nabla b|^{\beta} = f(t)\left|\mathbb{S}^{n-1}\right| \int_{0}^{e^{2 t}} s^{n-1-\alpha} d s \sim \left|\mathbb{S}^{n-1}\right|\frac{e^{2(n-\af)t}f(t)}{n-\af}.$$

Hence an exponentially decaying $f$ (in this case $f(t) = e^{-2(n-\alpha) t}$) is a natural and sharp choice.

\begin{proof}[Proof of Theorem \ref{thm2}]

For each fixed $\alpha \leq n$, denote for $\beta\geq \max \{n-\af, 0\}$,  $$\displaystyle V_{\beta}(t)=f(t) \int_{\Phi_{t}(B)} b^{-\alpha}|\nabla b|^{\beta} \, d\mathcal{H}^n.$$ Equivalently we may write $\displaystyle V_{\bt}(t) = f(t) \int_B (b \,\circ\,\Phi_t)^{-\af} ( |\nab b|^{\bt} \, \circ\, \Phi_t)|\det d\Phi_t|\, d\mathcal{H}^n$. Observe that $V_{\beta}$ is nonincreasing in $\beta$ since $|\nab b| \leq 1$ by Theorem \ref{thm:gradest}. 

We may calculate $dV_{\beta}(t)/dt$ as

$$
\begin{aligned}
\frac{dV_{\beta}(t)}{d t} &=f^{\prime}(t) \int_{\Phi_{t}(B)}b^{-\alpha}|\nabla b|^{\beta}+f(t) \frac{d}{d t} \int_{\Phi_{t}(B)} b^{-\alpha}|\nabla b|^{\beta} \,d\mathcal{H}^n\\
& =\frac{f^{\prime}(t)}{f(t)} V_{\beta}(t) +f(t)\underbrace{\int_{\Phi_t(B)} b^{-\alpha}|\nabla b|^{\beta}\Delta b^2 \,d\mathcal{H}^n}_{I}+f(t)\underbrace{\int_{\Phi_t(B)} \frac{d}{d t}\left(b^{-\alpha}|\nabla b|^{\beta}\right)\, d\mathcal{H}^n}_{II}. \\
\end{aligned}
$$

Let us calculate $I$ and $II$ separately. Since $\Delta b^2 = 2n|\nab b|^2$, 

$$I=2 n \int_{\Phi_{t}(B)}{b^{-\alpha}}|\nabla b|^{\beta+2}.$$

Hence,
$$f(t) I=2 n V_{\beta+2}(t).$$

Now we calculate $II$. To emphasize that we are calculating at points in $\Phi_t(B)$ and not $B$, we write $b$ as $b \circ \Phi_t$, $|\nab b|$ as $|\nab b| \circ \Phi_t$, etc. when it is clearer to do so.

$$
\begin{aligned}
\frac{d}{d t}\left(b \circ \Phi_{t}\right)^{-\alpha}&=-\alpha\left(b \circ\Phi_{t}\right)^{-\alpha-1}\left\langle\nabla b, \nabla b^{2}\right\rangle \\
& =-2 \alpha\left(b \circ \Phi_{t}\right)^{-\alpha}|\nabla b|^{2} \circ \Phi_{t} ,
\end{aligned}
$$
and
$$
\begin{aligned}
\frac{d}{d t}\left(|\nabla b| \circ \Phi_{t}\right)^{\beta}&=\frac{d}{d t}\left(|\nabla b|^{2}\circ \Phi_{t}\right)^{\frac{\beta}{2}} \\
&=\frac{\beta}{2}\left(|\nabla b|^{2} \circ \Phi_{t}\right)^{\frac{\beta}{2}-1}\left\langle\nabla|\nabla b|^{2}, \nabla b^{2}\right\rangle.
\end{aligned}
$$

Applying Lemma \ref{lem_nab_grad} and the assumption that $\bt \ge 0$ and $\nab^2 b^2 \leq Cg$, we conclude that

\begin{align}
\frac{d}{d t}\left(|\nabla b| \circ \Phi_{t}\right)^{\beta} &\leq \beta\left(|\nabla b|^{2} \circ \Phi_{t}\right)^{\frac{\beta}{2}-1}\left[C|\nabla b|^{2}-2|\nabla b|^{4}\right] \circ \Phi_{t} \label{equality} \\
&= \beta C (|\nab b|\circ \Phi_t)^{\beta} - 2\beta (|\nab b|\circ \Phi_t)^{\beta + 2}. \nonumber
\end{align}

Combining, we get at points in $\Phi_t(B)$,
$$
\begin{aligned}
\frac{d}{d t}& \left(b^{-\af} |\nab b|^{\bt}\right) \\
&\leq-2 \alpha b^{-\alpha}|\nabla b|^{\beta + 2} + b^{-\alpha}\left[\beta C|\nab b|^{\beta} - 2\beta |\nab b|^{\beta + 2} \right] \\
&=b^{-\alpha}|\nab b|^{\beta}\left[\beta C-2(\af+\bt) |\nab b|^2\right].
\end{aligned}
$$

Putting everything together, we have
$$
\begin{aligned}
& \frac{d}{d t} V_{\beta}(t)=f(t) I+f(t) II+\frac{f^{\prime}(t)}{f(t)} V_{\beta}(t) \\
& \leq 2nV_{\bt+2}(t) + \bt C V_{\bt}(t) - 2(\af + \bt)V_{\bt+2}(t)+\frac{f^{\prime}(t)}{f(t)} V_{\beta}(t) \\
&=\left(\bt C + \frac{f'(t)}{f(t)}\right)V_{\bt}(t) + 2\left(n-\af - \bt\right)V_{\bt+2}(t).
\end{aligned}
$$
By the assumptions that $\displaystyle \bt C + \frac{f'(t)}{f(t)} \leq 0$ and $n-\af - \bt \leq 0$, the monotonicity formula now follows. 

Equality holds in (\ref{equality}) when $\Hess_{b^2} = Cg$ on $\Phi_t(B)$. Taking the trace, we obtain that $|\nab b|$ is a constant. Since $|\nab b| \to 1$ as one approaches $p$, $|\nab b|\equiv 1$ on $\Phi_t(B)$. Then we can conclude that $M$ is isometric to the Euclidean space $\R^n$ by the condition for equality in Theorem \ref{thm:gradest}.
\end{proof}

\begin{proof}[Proof of Theorem \ref{thm1}] Theorem \ref{thm1} is now immediate by taking $f(t) = e^{-C\bt t}$ in the proof above.
\end{proof}

\section{Proof of the second monotonicity formula}

By the divergence theorem,
\begin{align*}
\int_{b \leq r} \langle \nab |\nab b|^2, \nab b^2 \rangle	&= -\int_{b \leq r} |\nab b|^2\Delta b^2 + \int_{b=r}|\nab b|^2 \langle \frac{\nab b}{|\nab b|}, \nab b^2 \rangle \\
&= -\int_{b \leq r} 2n|\nab b|^4 + \int_{b=r} 2b|\nab b|^3. \\ 
\end{align*}

On the other hand, by Lemma \ref{lem_nab_grad} and since $\nab^2 b^2 \leq Cg$,
$$
\int_{b \leq r} \langle \nab |\nab b|^2, \nab b^2 \rangle \leq 	\int_{b\leq r} 2C|\nab b|^2 - 4\int_{b\leq r}|\nab b|^4.
$$

Combining, we have

\begin{equation} \label{eq1}
\int_{b=r} 2b|\nab b|^3 \leq \int_{b\leq r} 2C|\nab b|^2 +(2n- 4)\int_{b\leq r}|\nab b|^4.	
\end{equation} 

Now let us define
$$V(r) = \int_{b \leq r} |\nab b|^4$$

so that the coarea formula yields
$$V'(r) = \int_{b=r} |\nab b|^3. $$

Then we may rewrite (\ref{eq1}) as

$$r V'(r) \leq \frac{C|\mathbb{S}^{n-1}|}{n} r^n + (n-2)V(r).$$

Here, we used the identity that for any $r > 0$, $\int_{b \leq r} |\nab b|^2 = \frac{|\mathbb{S}^{n-1}|}{n} r^n $. Equivalently, \begin{equation} \label{eqn_thm3}
\left(r^{2-n} V(r) \right)' \leq \frac{C|\mathbb{S}^{n-1}|}{n}r.	
\end{equation}

This yields Theorem \ref{thm3}. Now equality holds at some $r$ if and only if $\Hess_{b^2} = Cg$ on $\{b\leq r\}$. Since $p \in \{b \leq r\}$, by the same reasoning as before we can conclude that $M$ is isometric to $\R^n$.

\begin{remark}
By integrating equation (\ref{eqn_thm3}), we get as a corollary that $V(r) \leq \frac{C|\mathbb{S}^{n-1}|}{2n}r^n$. Although the at-most-$r^n$ growth of $V$ already follows from the definition and Theorem \ref{thm:gradest}, in this expression it becomes clearer how the constant is determined by $\Hess_{b^2}$.
	
\end{remark}

\section{examples}

In this section we investigate two explicit examples on which Assumption ($\ast$) is satisfied.

\begin{ex}[Small perturbation of cones] It is straightforward to check that on a cone $g_C=dr^2 + ar^2 g_N$ where $a>0$ is a constant and the cross section $N$ is an $(n-1)$-dimensional compact Riemannian manifold, the function $r^{n-2}$ is harmonic and satisfies $\Hess_{r^2} =2g$. The Hessian bound persists on a Riemannian manifold that is a sufficiently small $\mathcal{C}^2$ perturbation of such a cone. For instance, if $(M,g) = \left([0, \infty) \times N, dr^2 + f^2 g_{N}\right)$ with $f:M \to [0, \infty)$ satisfying $|\nab^k (g - g_C)| = o(r^{-k+n})$ for $k=0,1,2$ for both $r \to 0$ and $r\to \infty$, then the Green function is uniformly $\mathcal{C}^2$-close to the solution on the cone. Hence Assumption $(\ast)$ is satisfied.
\end{ex}

\begin{ex}[Warped products]
We examine which nonparabolic warped products satisfy Assumption ($\ast$). Consider a metric of the form $dr^2 + f(r)^2 g_N$ where $(N, g_N)$ is a smooth Riemannian manifold. We will conclude that if the warped product is nonpabolic and $f$ has sublinear growth up to second order, then Assumption ($\ast$) holds true, by a calculation which we lay out below (together with the precise meaning of sublinear growth in this context).

 Using the radial expression of the Laplacian, the Green function $G$ on this warped product is a radial function that satisfies \begin{equation} \label{G_eqn}
	\Delta G = \frac{\pr^2}{\pr r^2} G + \frac{n-1}{f(r)} \frac{\pr}{\pr r} G = 0.
\end{equation} 

Let $F(r) = f(r)^2$. For $b(r) = G(r)^{\frac{1}{2-n}}$ as before, let $h: \R \to \R$ be such that $b(r)^2 = h(F(r))$. We denote the derivative with respect to $r$ with the usual prime, and the derivative with respect to $F$ with an upper dot. These functions can be understood both as a function of $r$ or $F$, or a function on $M$ by a composition with $r$. Then $$\nab b^2 = \overset{.}{h}(F) \nab F,$$ and

\begin{equation}\label{eq:Hessb}
\nab^2 b^2 = \overset{..}{h}(F) \, dF \otimes dF + \overset{.}{h}(F)\nab^2 F.	
\end{equation}

Let us calculate the right hand side of (\ref{eq:Hessb}). Since $F$ is radial we have that 
$$dF = F'(r) dr. $$

To calculate $\nab^2 F$, one can use the well-known elementary identity on a warped product that $\nab^2 (\int f \, dr) = f' g$, which translates into
$$\nab^2 F = \left(F'' - 2(f')^2 \right) dr \otimes dr + 2(f')^2g.$$

Now (\ref{eq:Hessb}) becomes
\begin{equation} \label{eqn_Hess_warp}
\nab^2 b^2 = \overset{..}{h}(F)(F')^2 dr \otimes dr +  \overset{.}{h}(F)\left[ \left(F'' - 2(f')^2 \right) dr \otimes dr + 2(f')^2g\right].	
\end{equation}

Therefore, if $|\overset{..}{h}(F)(F')^2+\overset{.}{h}(F) \left(F'' - 2(f')^2 \right)|$ and $|\overset{.}{h}(F)(f')^2|$ are bounded, then the warped product satisfies the desired pointwise Hessian estimate of $b^2$. Before setting out to estimate these two quantities, we note that in the model conical case of $f(r)=r$ and $h(F)=F$, the first term is 0 and the second is 2.

Inserting $G = h(F)^{\frac{2-n}{2}}$ into (\ref{G_eqn}) gives \begin{equation} \label{hF_eqn}
	-\frac{n}{2} \left[\overset{.}{h}(F) F' \right]^2 + h(F) \left[\overset{..}{h}(F) (F')^2 + \frac{n-1}{\sqrt{F}} \overset{.}{h}(F) F' +\overset{.}{h}(F) F'' \right] = 0.
\end{equation}

(\ref{hF_eqn}) is equivalent to 
\begin{equation} \label{hF_eqn2}
	\overset{.}{\left(h^{-n/2}\overset{.}{h}\right)} = -\left(\frac{n-1}{\sqrt{F}F'} +\frac{F''}{(F')^2} \right)h^{-n/2}\overset{.}{h}.
\end{equation}

Integrating (\ref{hF_eqn2}) twice gives that for constants $C_1, C_2$,
\begin{equation} \label{h_calculated}
	h = \left[C_1 \int \exp\left(-\int\left(\frac{n-1}{\sqrt{F}} + \frac{F''}{F'}\right) \, dr \right) \, dF + C_2 \right]^{\frac{2}{2-n}},
\end{equation}

\begin{equation} \label{hdot_calculated}
\overset{.}{h} =\frac{2}{2-n} C_1 h^{n/2}  \exp\left(-\int\left(\frac{n-1}{\sqrt{F}} + \frac{F''}{F'}\right) \, dr \right).
\end{equation}

Rearranging (\ref{hF_eqn}) gives

\begin{equation} \label{hdotdot_calculated}
\overset{..}{h} = \left[\frac{n\overset{.}{h}}{2h}-\left(\frac{n-1}{\sqrt{F}F'} + \frac{F''}{(F')^2} \right) \right] \overset{.}{h}.
\end{equation}

Let us now focus on the case of $f$ with polynomial growth up to second order, that is, where there exist an $\af\in \R$ and some $C_0>0$ so that $$\frac{r^{\af-i}}{C_0} <|\nab^i f| < C_0r^{\af-i} \text{ for all } i=0,1,2.$$ From here on, we also introduce the notations $C_3, C_4, \cdots$ in the coefficients and exponents, which are bounded quantities but not necessarily constant and may vary from line to line. This suffices for our purposes of only tracking the growth of the functions as $r\to \infty$, not their precise values.

Inserting the polynomial growth of $f$ into (\ref{h_calculated}), (\ref{hdot_calculated}), (\ref{hdotdot_calculated}), for $\af \neq 1$ we can bound $h,\overset{.}{h},\overset{..}{h}$ as

$$h \sim \left[C_3 \int r^{C_4} e^{-C_5 r^{1-\af}}\, dr + C_2 \right]^{\frac{2}{2-n}},$$
$$\overset{.}{h} \sim \left[C_3\int r^{C_4} e^{-C_5 r^{1-\af}}\, dr + C_2 \right]^{\frac{n}{2-n}}r^{C_4}e^{-C_5 r^{1-\af}} ,$$
$$\overset{..}{h} \sim \left[ \frac{n}{2h} - \frac{C_6}{r^{3\af-1}}- \frac{C_7}{r^{2\af}}\right] \overset{.}{h}$$
where $C_5, C_6>0$. Now in order to bound $|\overset{..}{h}(F)(F')^2+\overset{.}{h}(F) \left(F'' - 2(f')^2 \right)|$ and $|\overset{.}{h}(F)(f')^2|$ as $r\to \infty$, it suffices to bound  $\overset{..}{h}(F)(F')^2, \overset{.}{h}(F)F''$, and $\overset{.}{h}(F)(f')^2$. For this we only need to bound $\overset{..}{h}r^{4\af-2}$ and $\overset{.}{h}r^{2\af-2}$. If $\af <1$, then the integral $\int r^{C_4} e^{-C_5 r^{1-\af}}\, dr$ is convergent and $e^{-C_5 r^{1-\af}}$ decays exponentially, so $\overset{.}{h}$ and $\overset{.}{h}r^{2\af-2}$ decays exponentially and $h$ is bounded. Therefore $\overset{..}{h}r^{4\af-2}$ also decays exponentially.

We briefly also mention the critical case of $\af = 1$ which is when $f$ has linear growth up to second order. By a similar brute force estimate, we obtain (for $C_i$ not necessarily the same as before)
$$h \sim [C_1r^{-C_3} + C_2]^{\frac{2}{2-n}},$$
$$\overset{.}{h} \sim r^{C_4},$$
$$\overset{..}{h} \sim r^{C_4}$$
for some $C_3, C_4$. It is possible to track the $C_i$'s in a precise manner if $C_0$ is sufficiently close to 1, or equivalently, $F/r^2, F'/r, F''$ are sufficiently close to the model case where $F(r)=r^2$. In that case, for example $\overset{.}{h}$ can be bounded with $$\overset{.}{h} \sim r^{n\left(\frac{(1+\ep)n-2}{n-2} - (1-\ep) \right)}$$ where $\ep \to 0$ as we approach the model conical case. This is not strong enough to guarantee that $\overset{..}{h}r^{4\af-2}$, $\overset{.}{h}r^{2\af-2}$ stay bounded, but at least we can estimate their growth.

To summarize, we have established that on a nonparabolic warped product $dr^2 + f(r)^2 g_N$ with $f$ having sublinear growth up to second order, Assumption $(\ast)$ is satisfied. Moreover, $\nab^2 b^2 \leq H(r) g$ for some function $H$ decaying exponentially in $r$. In proving the exponential decay, we derived equations (\ref{h_calculated}), (\ref{hdot_calculated}), (\ref{hdotdot_calculated}) which together with (\ref{eqn_Hess_warp}) would yield a precise estimate of the Hessian of the Green function on any warped product. 
\end{ex}


\begin{thebibliography}{4}

\bibitem{AFM} Agostiniani, V., Fogagnolo, M., Mazzieri, L., {\em Sharp geometric inequalities for closed hypersurfaces in manifolds with nonnegative ricci curvature}, Invent. Math., \textbf{222} (2020), 1033--1101.

\bibitem{AFM2} -----, {\em Minkowski inequalities via nonlinear potential theory}, Arch. Ration. Mech. Anal., \textbf{244} (2022), 51--85.


\bibitem{AM} Agostiniani, V., Mazzieri, L., {\em On the geometry of the level sets of bounded static potentials}, Comm. Math. Phys., \textbf{355} (2017), 261--301.

\bibitem{AMO}  Agostiniani, V., Mazzieri, L., Oronzio, F., {\em A Green’s function proof of the positive mass theorem}, Comm. Math. Phys., \textbf{405} (2024), article 54.


\bibitem{CCLT} Chan, P.-Y., Chu, J., Lee, M.-C., Tsang, T.-Y., {\em Monotonicity of the $p$-Green functions},  Int. Math. Res. Not. IMRN, \textbf{2024}(9) (2024), 7998--8025.



\bibitem{CL} Chodosh, O., Chao, L., {\em Stable minimal hypersurfaces in $\R^4$}. Preprint, arXiv:2108.11462.

\bibitem{C} Colding, T.H., {\em New monotonicity formulas for Ricci curvature and applications. I}, Acta Math., \textbf{209} (2012), 229--263.


\bibitem{CM2} Colding, T.H., Minicozzi, W.P., II,, {\em Ricci curvature and monotonicity for harmonic functions}, Calc. Var. Partial Differential Equations, \textbf{49}(3-4) (2014), 1045--1059.

\bibitem{CM3} -----, {\em On uniqueness of tangent cones of Einstein manifolds}, Invent. Math., \textbf{196} (2014), 515--588.







\bibitem{MRS} Mari, L., Rigoli, M., Setti, A., {\em On the $1/H$-flow by $p$-Laplace approximation: new estimates via fake distances under Ricci lower bounds}. Preprint, arXiv:1905.00216.

\bibitem{MW} Munteanu, O., Wang, J., {\em Comparison theorems for 3D manifolds with scalar curvature bound}. Int. Math. Res. Not. IMRN, \textbf{2023}(3) (2023), 2215--2242.

\bibitem{MW2} -----, {\em Geometry of three-dimensional manifolds with positive scalar curvature}. Preprint, arXiv:2201.05595.


\bibitem{N2} Ni, L., {\em Monotonicity and Li-Yau-Hamilton inequalities}, Surv. Differ. Geom., \textbf{XII} (2008), 251--301.

\bibitem{P} Park, J., {\em Matrix inequality for the Laplace equation}, Int. Math. Res. Not. IMRN, \textbf{2019}(11) (2019), 3485--3497.

\bibitem{Pog} Pogessi, G., {\em Radial symmetry for $p$-harmonic functions in exterior and punctured domains}, Appl. Anal., \textbf{98}(10) (2019), 1785--1798.

\bibitem{Po} Poon, C.C., {\em Unique continuation for parabolic equations}, Comm. PDEs, \textbf{21}(3--4) (1996), 521--539.

\bibitem{V} Varopoulos, N.T., {\em The Poisson kernel on positively curved manifolds}, J. Funct. Anal., \textbf{44}(3) (1981), 359--380.

\end{thebibliography}
\end{document}